\documentclass[preprint,12pt]{elsarticle}
\usepackage{}

\usepackage{amsfonts}
\usepackage{amsmath}
\usepackage{amssymb}
\usepackage{graphics}
\usepackage{graphicx}
\usepackage{mathrsfs}
\usepackage{indentfirst}
\usepackage{stmaryrd}
\usepackage{latexsym}
\usepackage{xypic}
\newtheorem{definition}{Definition}[section]

\newtheorem{theorem}[definition]{Theorem}
\newtheorem{proposition}[definition]{Proposition}
\newtheorem{example}[definition]{Example}

\newtheorem{remark}[definition]{Remark}

\newproof{proof}{\textbf{Proof}}

\journal{}

\begin{document}

\begin{frontmatter}



\title{\textbf{Stabilizers in MTL-algebras}}


\author{Jun Tao Wang$^{\ast,a}$, Peng Fei He$^{b}$, Arsham Borumand Saeid$^{c}$}
\cortext[cor1]{Corresponding author. \\
Email addresses:  wjt@stumail.nwu.edu.cn(J.T. Wang), hepengf1986@126.com (P.F. He),\\ arsham@uk.ac.ir (A. Borumand. Saeid).}
\address[A]{School of Mathematics, Northwest University, Xi'an, 710127, P.R. China}
\address[B]{School of Mathematics and Information Science, Shaanxi Normal University, Xi'an, 710119, P.R. China}
\address[C]{Department of Pure Mathematics, Faculty of Mathematics and Computer, Shahid Bahonar University of Kerman, Kerman, Iran}

\begin{abstract}

 In the paper, we introduce some stabilizers and investigate related properties of them in MTL-algebras.
 Then, we also characterize some special classes of MTL-algebras, for example, IMTL-algebras, integral MTL-algebras, G\"{o}del algebras and MV-algebras, in terms of these stabilizers.
 Moreover, we discuss the relation between stabilizers and several special filters (ideals) in MTL-algebras.
 Finally, we discuss the relation between these stabilizers and prove that the right implicative stabilizer and right multiplicative stabilizer are order isomorphic. This results also give answers to some open problems,
 which were proposed by Motamed and Torkzadeh in [Soft Comput, {\bf 21} (2017) 686-693].
\end{abstract}


\end{frontmatter}

\section{Introduction}
\label{intro}Much of human reasoning and decision making is based on an environment of imprecision, uncertainty, incompleteness of information, partiality of truth and partiality of possibility-in short, on an environment of imperfect information. Hence how to represent and simulate human reasoning become a crucial problem in information science field. For this reason, various kinds of fuzzy logical algebras as the semantical systems of fuzzy logic systems have been extensively introduced and studied, for example, MV-algebras \cite{Chang}, BL-algebras \cite{Hajek}, MTL-algebras \cite{Esteva}, NM-algebras \cite{Esteva}. Among these logical algebras, MTL-algebras are the most significant because the others are all particular cases of MTL-algebras. As a more general residuated structure based continuous t-norm and its residua, an MTL-algebra is a BL-algebra without the divisibility. In fact, MTL-algebras contain all algebras induced by left-continuous t-norm and its residua. Therefore, MTL-algebras play an important role in studying fuzzy logics and their related structures.


The notion of stabilizers, introduced from fixed point set theory, is helpful for studying structures and properties in algebraic systems. In fact, stabilizer as a part of a monoid acting on a nonempty set and the stabilizer of nonempty subset $X$ is the transporter of $X$ to itself. From a logic point of view, stabilizer can be used in studying the consequence operators in the correspondence logic system. Since stabilizer was successful in several distinct tasks in various branches of mathematics \cite{Roudabri}, it has been extended to various logical algebras, for example, Haveshki was first introduced the stabilizers in BL-algebras and investigate some basic properties of them. Also, they discuss the relations between stabilizers and filters in BL-algebras in \cite{Haveshki1}. Inspired by this, Borzooei introduced some new types of stabilizers and determined the relations among stabilizers in BL-algebras, they also show that the (semi) normal filters and fantastic filters are equal in BL-algebras via stabilizers in \cite{Borzooei1}. After then, Saeid introduced two kinds of stabilizers and discussed the relation between stabilizers and some other ideals in MV-algebras, they also prove that the lattices of ideals of MV-algebras forms a pseudocomplement lattice via stabilizers in \cite{Saeid2}. Recently, Motamed has introduced the notion of right stabilizers in BL-algebras and two class of BL-algebras, called RS-BL-algebras and semi RS-BL-algebras, and has discussed the relations between them and (semi)local BL-algebras, they also proposed some open problems related to stabilizers in \cite{Somayeh}, for example, ``Let $X$ be a nonempty subset of a BL-algebra $L$. Is $X_r\cup\{0\}$ a subalgebra of $L$?" and ``If $L$ is a RS-BL-algebra and $F$ is an any filter of $L$, then $(F_r)_r=F$?". After then, Turunen proved that RS-BL-algebras are equivalent to MV-algebras in \cite{turunen}.

In this paper, we will study stabilizers on MTL-algebras. One of our aims is to give answers to serval open problems related to stabilizers in BL-algebras in \cite{Somayeh}. On the other hand, the main focus of existing research about stabilizers on MV-algebras, BL-algebras, etc. All the above-mentioned algebraic structures satisfy the divisibility condition $x\wedge y=x\odot(x\rightarrow y)$. In this case, the conjunction $\odot$ on the unit interval corresponds to a continuous t-norm. However, there are few research about the stabilizer on residuated structures without the divisibility condition so far. Therefore, it is meaningful to study stabilizers in MTL-algebras for providing a solid algebraic foundation for consequence operations in MTL logic. This is the motivation for us to investigate stabilizer theory on MTL-algebras.

 This paper is structured in four sections. In order to make the paper as self-contained as possible, we recapitulate in Section 2 the definition of MTL-algebras, and review their basic properties. In Section 3, we introduce implicative stabilizers and characterize some special classes of MTL-algebras in terms of these stabilizers. In Section 4, we introduce multiplicative stabilizers and investigate related properties of them. Using multiplicative stabilizers, we give some characterizations of G\"{o}del algebras and linearly order G\"{o}del algebras. Finally, we discuss the relations between implicative stabilizers and  multiplicative stabilizers.

\section{Preliminaries}
In this section, we summarize some definitions and results about MTL-algebras, which will be used in this paper.
\begin{definition}\emph{\cite{Esteva} An algebraic structure $(L,\wedge,\vee,\odot,\rightarrow,0,1)$ of type $(2,2,2,2,0,0)$ is called an \emph{MTL-algebra} if it satisfies the following conditions:}
\begin{enumerate}[(1)]
  \item \emph{$(L,\wedge,\vee,0,1)$ is a bounded lattice,
  \item $(L,\odot,1)$ is a commutative monoid,
  \item $x\odot y\leq z$ if and only if $x\leq y\rightarrow z$,
  \item $(x\rightarrow y)\vee(y\rightarrow x)=1$,
  for any $x,y,z\in L$.}
\end{enumerate}
\end{definition}

In what follows, by $L$ we denote the universe of an MTL-algebra $(L,\wedge,\vee,\odot,\rightarrow,0,1)$.
For any $x\in L$ and a natural number $n$, we define $\neg x=x\rightarrow 0$, $\neg\neg x=\neg(\neg x)$, $x^0=1$ and $x^n=x^{n-1}\odot x$ for all $n\geq 1$.
\begin{proposition}\emph{\cite{Noguera} In any MTL-algebra $L$, the following properties hold: }
\begin{enumerate}[(1)]
  \item \emph{$x\leq y$ if and only if $x\rightarrow y=1$,
  \item $x\odot y\leq x\wedge y$,
  \item $x\rightarrow (y\wedge z)=(x\rightarrow y)\wedge(x\rightarrow z)$,
  \item $(x\vee y)\rightarrow z=(x\rightarrow z)\wedge(y\rightarrow z)$,
  \item $x\rightarrow y=x\rightarrow(x\wedge y)$,
  \item $x\rightarrow y=(x\vee y)\rightarrow y$,
  \item $x\wedge y\rightarrow z=(x\rightarrow z)\vee(y\rightarrow z)$,
  \item $x\vee y=((x\rightarrow y)\rightarrow y)\wedge((y\rightarrow x)\rightarrow x)$,
  \item $x\leq y\rightarrow x$,
 \item $\neg x=\neg\neg\neg x$, for all $x,y,z\in L$. }
\end{enumerate}
\end{proposition}

\begin{definition}\emph{\cite{Noguera} Let $L$ be an MTL-algebra. Then $L$ is called:}
\begin{enumerate}[(1)]
  \item \emph{a\emph{ BL-algebra} if $x\wedge y=x\odot(x\rightarrow y)$ for any $x,y\in L$,
  \item an\emph{ MV-algebra} if $(x\rightarrow y)\rightarrow y=(y\rightarrow x)\rightarrow x$ for any $x,y\in L$,
  \item a\emph{ G\"{o}del algebra} if $x\odot y=x\wedge y=x\odot(x\rightarrow y)$ for any $x,y\in L$,
  \item an\emph{ IMTL-algebra} if $\neg\neg x=x$ for any $x\in L$,
   \item an\emph{ integral MTL-algebra} if $x\odot y=0$, then $x=0$ or $y=0$ for any $x,y\in L$.}
\end{enumerate}
\end{definition}

A nonempty subset $F$ of $L$ is called a \emph{filter} of $L$ if it satisfies: (1) $x,y\in F$ implies $x\odot y\in F$; (2) $x\in F$, $y\in L$ and $x\leq y$ implies $y\in F$. We denote by $F[L]$ the set of all filers of $L$. A filter $F$ of $L$ is called a \emph{proper filter} if $F\neq L$. A proper filter $F$ of $L$ is called a \emph{prime filter} if for each $x,y\in L$, $x\vee y\in F$, implies $x\in F$ or $y\in F$. If $X$ is a nonempty subset of $L$, then we denote the filter generated by $X$ by $\langle X\rangle$. Clearly, we have $\langle X\rangle=\{x\in L|x\geq x_1\odot x_2\odot \cdots \odot x_n, x_i\in X\}=\{x\in L|x_1\rightarrow(x_2\rightarrow(\cdots(x_n\rightarrow x)\cdots))=1,x_i\in X\}$, see \cite{Esteva,Haveshki,Borzooei}.\\



If $(L,\wedge,\vee,\rightarrow,\odot,0,1)$ is an MTL-algebra, we denote by $G(L)$ the set of all idempotent elements of $(L,\odot,1)$. The set $G(L)$ is the universe of a G\"{o}del subalgebra of $L$, which is called the \emph{G\"{o}del center} of $L$\cite{Kowalski}.

\begin{proposition}\emph{\cite{Kowalski} Let $L$ be an MTL-algebra. For every $x,y\in L$ and $e\in G(L)$, we have:}
\begin{enumerate}[(1)]
  \item \emph{$e\odot e=e$,
  \item $e\odot(x\rightarrow y)=e\odot[(e\odot x)\rightarrow(e\odot y)]$.}
\end{enumerate}
\end{proposition}

A nonempty subset $I$ of an MTL-algebra $(L,\wedge,\vee,\rightarrow,\odot,0,1)$ is called a \emph{lattice
ideal} of $L$ if it satisfies: (i) for all $x,y\in I$, $x\vee y\in I$; (ii) for all $x,y\in L$, if $x\in I$
and $y\leq x$, then $y\in I$. That is, a lattice ideal of an MTL-algebra $L$ is the
notion of ideal in the underlying lattice $(L,\wedge,\vee)$. A lattice ideal $I$ of $L$ is called to be \emph{prime} if it satisfies for all $x,y\in L$, $x\wedge y\in I$
implies $x\in I$ or $y\in I$. For any nonempty subset $H$ of $L$, the smallest lattice ideal containing $H$ is
called the lattice ideal generated by $H$. The lattice ideal generated by $H$ will be
denoted by $(H]$. In particular, if $H =\{t\}$, we write $(t]$ for $(\{t\}]$, $(t]$ is called a
\emph{principal lattice ideal} of $L$. It is easy to check that $(t] =\downarrow t =\{x\in L|x\leq t\}$\cite {Gratzer}.
\section{Implicative stabilizers in MTL-algebras}

In the section, we investigate left and right implicative stabilizers and discuss the relation between them. Then, we give some characterizations of IMTL-algebras, integral MTL-algebras and MV-algebras via implicative stabilizers.

\begin{definition}\emph{Let $L$ be an MTL-algebra and $X$ be a nonempty subset of $L$. The \emph{left and right implicative stabilizer} of $X$ are defined as follows, respectively,}
\begin{center}$X_l=\{a\in L|a\rightarrow x=x$, \emph{for all} $x\in X\}$,\\
$X_r=\{a\in L|x\rightarrow a=a$, \emph{for all} $x\in X\}$.
\end{center}

\emph{The set $X_s=X_l\cap X_r$ is called the \emph{implicative stabilizer} of $X$. For convenience, the implicative stabilizer, left implicative stabilizer and right implicative stabilizer of $X=\{x\}$ are denoted by $S_x$, $L_x$ and $R_x$, respectively.}
\end{definition}

Now, we present some examples for implicative stabilizers in an MTL-algebra.
\begin{example}\emph{Let $L=\{0,a,b,1\}$ with $0\leq a\leq b\leq 1$. Consider the operations $\odot$ and $\rightarrow$ given by the following tables}:\\
\begin{center}
\begin{tabular}{c|cccc}

  $\odot$ & $0$ & $a$ & $b$ & $1$ \\
   \hline
  $0$ & $0$ & $0$ & $0$ & $0$ \\
  $a$ & $0$ & $0$ & $0$ & $a$ \\
  $b$ & $0$ & $0$ & $b$ & $b$ \\
  $1$ & $0$ & $a$ & $b$ & $1$ \\
\end{tabular}{\qquad}
\begin{tabular}{c|cccc}

  $\rightarrow$ & $0$ & $a$ & $b$ & $1$ \\
   \hline
  $0$ & $1$ & $1$ & $1$ & $1$ \\
  $a$ & $b$ & $1$ & $1$ & $1$ \\
  $b$ & $a$ & $a$ & $1$ & $1$ \\
  $1$ & $0$ & $a$ & $b$ & $1$ \\
\end{tabular}
\end{center}

\emph{Then $(L,\wedge,\vee,\odot,\rightarrow,0,1)$ is an MTL-algebra. If we put $X=\{b\}$, then $X_l=\{1\}\neq X_r=\{a,1\}$ and hence $X_s=\{1\}$, in this case, $X_r$ is not a filter of $L$}.
\end{example}

Example 3.2 shows that $X_r\neq X_l$ and $X_r$ is not a filter of an MTL-algebra, in general.

\begin{example}\emph{ Let $L=[0,1]$. For any $x,y\in L$, we define $x\wedge y=$min$\{x,y\}$,$x\vee y=$max$\{x,y\}$, $\neg x=1-x$, $x\rightarrow y=1$ if $x\leq y$; otherwise $x\rightarrow y=\neg x\vee y$, and $x\odot y=0$ if $x\leq \neg y$; otherwise $x\wedge y$.
Then $(L,\wedge,\vee,\odot,\rightarrow,0,1)$ is an MTL-algebra, $R_\frac{2}{3}=[\frac{1}{3},\frac{2}{3})$, $L_\frac{2}{3}=(\frac{2}{3},1]$, so $S_\frac{2}{3}=\emptyset$.}
\end{example}

The following proposition provides some useful properties of implicative stabilizers in an MTL-algebra.

\begin{proposition} \emph{Let $L$ be an MTL-algebra and $X,Y$ be two nonempty subsets of $L$. Then the following properties hold:}
\begin{enumerate}[(1)]
   \item \emph{$X_r=\cap_{x\in X}R_x$, $X_l=\cap_{x\in X}L_x$ and $X_s=\cap_{x\in X}S_x$,
  \item if $X\subseteq Y$, then $Y_r\subseteq X_r$, $Y_l\subseteq X_l$, and $Y_s\subseteq X_s$,
  \item $\langle X\rangle_r=X_r$,
  \item $X=\{1\}$ if and only if $X_l=X_r=X_s=L$,
  \item $L_r=L_l=L_s=\{1\}$,
  \item $R_0=\{1\}$ and so $S_0=\{1\}$,
  \item if $a,b\in X_r$, then $a\wedge b,a\rightarrow b\in X_r$, and so $a\vee b\in X_r$,
  \item $X_l$ is a filter of $L$,
  \item $\langle X\rangle \cap X_r=\{1\}=\langle X\rangle \cap X_s$.}
\end{enumerate}
\end{proposition}
\begin{proof} The proofs of $(1),(2),(4),(5),(6),(8)$ and $(9)$ are easy.

$(3)$  Since $X\subseteq \langle X\rangle$, by (2), we have $\langle X\rangle_r\subseteq X_r$. On the other hand, suppose that $a\in X_r$, and so $x\rightarrow a=a$, for all $x\in X$.  For any $y\in \langle X\rangle$, there exist $y_1,y_2\cdots y_n\in X$ such that $y_1\rightarrow(y_2\rightarrow(\cdots (y_n\rightarrow y))=1$, hence $y\rightarrow a\leq y_1\rightarrow(y_2\rightarrow\cdots y_n)\rightarrow a$. Thus, $a=y_1\rightarrow a=y_1\rightarrow(y_2\rightarrow a)\cdots y_1\rightarrow(y_2\rightarrow y_3\rightarrow\cdots(y_n\rightarrow a)))\leq y\rightarrow a$. On the other hand, we have $y\rightarrow a\leq a$ for all $y\in\langle X\rangle$, and so $a\in \langle X\rangle_r$, that is, $\langle X\rangle_r=X_r$.

$(7)$ For all $a,b\in X_r$, from Proposition 2.2(3), we have $x\rightarrow(a\wedge b)=(x\rightarrow a)\wedge (x\rightarrow b)=a\wedge b$, that is, $a\wedge b\in X_r$. Now, we will show that for any $a,b\in X_r$, $a\rightarrow b\in X_r$. Since $a\rightarrow b=(x\rightarrow a)\rightarrow(x\rightarrow b)=x\rightarrow((x\rightarrow a)\rightarrow b)=x\rightarrow(a\rightarrow b)$ and hence $a\rightarrow b\in X_r$. Also, from Proposition 2.2(8), one can prove that $a\vee b\in X_r$, for any $a,b\in X_r$.
\end{proof}

 The following example shows $X_l\neq \langle X\rangle_l$, for any nonempty subset $X$ of $L$, not holds, in general.

\begin{example}\emph{Let $L=\{0,a,b,c,1\}$ with $0\leq c\leq a,b\leq 1$. Consider the operations $\odot$ and $\rightarrow$ given by the following tables}:\\
\begin{center}
\begin{tabular}{c|ccccc}

  $\odot$ & $0$ & $a$ & $b$ & $c$ & $1$ \\
   \hline
  $0$ & $0$ & $0$ & $0$ & $0$ & $0$ \\
  $a$ & $0$ & $c$ & $a$ & $c$ & $a$ \\
  $b$ & $0$ & $c$ & $c$ & $b$ & $b$ \\
  $c$ & $0$ & $c$ & $c$ & $c$ & $c$ \\
  $1$ & $0$ & $a$ & $b$ & $c$ & $1$ \\
\end{tabular}{\qquad}
\begin{tabular}{c|ccccc}

  $\rightarrow$ & $0$ & $a$ & $b$ & $c$ & $1$ \\
   \hline
  $0$ & $1$ & $1$ & $1$ & $1$ & $1$ \\
  $a$ & $0$ & $b$ & $1$ & $b$ & $1$ \\
  $b$ & $0$ & $a$ & $a$ & $1$ & $1$ \\
  $c$ & $0$ & $1$ & $1$ & $1$ & $1$ \\
  $1$ & $0$ & $a$ & $b$ & $c$ & $1$ \\
\end{tabular}
\end{center}

\emph{Then $(L,\wedge,\vee,\odot,\rightarrow,0,1)$ is an MTL-algebra. If $X=\{b\}$, then $X_l=\{a,1\}$ and $\langle X\rangle_l=\{1\}$. Hence $X_l\neq \langle X\rangle_l$.}
\end{example}


 The following theorem shows $X_l=\langle X\rangle_l$, for any nonempty subset $X$, under certain conditions in an MTL-algebra.

\begin{theorem}\emph{Let $L$ be an MTL-algebra. Then the following conditions are equivalent:}
\begin{enumerate}[(1)]
  \item \emph{$X_l=\langle X\rangle_l$, for any nonempty subset $X$ of $L$,
  \item $a\rightarrow b=b$ if and only if $b\rightarrow a=a$, for any $a,b\in L$,
  \item $X_r$ is a filter of $L$, for any nonempty subset $X$ of $L$,
  \item $S_x=R_x=L_x$, for all $x\in L$,
  \item $X_r=X_l=X_s$, for all nonempty subset $X$ of $L$.}
\end{enumerate}
\end{theorem}
\begin{proof}$(1)\Rightarrow (2)$ Let $a,b\in L$ be such that $a\rightarrow b=b$, that is, $a\in L_b$. Then we have $a\in \langle b\rangle_l$ by (1). Since $((b\rightarrow a)\rightarrow a\in \langle b\rangle$, we have $a\in L_{(b\rightarrow a)\rightarrow a}$ and $1=(b\rightarrow a)\rightarrow (a\rightarrow a)=a\rightarrow[(b\rightarrow a)\rightarrow a]=(b\rightarrow a)\rightarrow a$, and finally $b\rightarrow a=a$. Similarly, we can prove that $b\rightarrow a=a$ implies $a\rightarrow b=a$ for any $a,b\in L$.

$(2)\Rightarrow (3)$ From (2), we have $R_x=\{a\in L|x\rightarrow a=a\}=\{a\in L|a\rightarrow x=x\}=L_x$ for any $x\in L$. From Proposition 3.4(8), we know that $L_x$ is a filter of $L$ and hence $R_x$ is a filter of $L$, for any $x\in L$.

$(3)\Rightarrow (4)$ By hypothesis, for any $x\in L$, $R_x$ is a filter of $L$. Let $a\in R_x$. Since $a\leq(a\rightarrow x)\rightarrow x$ and hence $(a\rightarrow x)\rightarrow x\in R_x$. Also, $x\leq (a\rightarrow x)\rightarrow x$, for $x \in\langle x\rangle$, and so $(a\rightarrow x)\rightarrow x\in R_x\cap \langle x\rangle=\{1\}$, by Proposition 3.4(9). Therefore, $a\rightarrow x=x$, for all $x\in X$, that is, $a\in R_x\cap L_x=S_x$, which implies $R_x\subseteq S_x$. On the other hand, we have $S_x\subseteq R_x$, thus $S_x=R_x=L_x$.

$(4)\Leftrightarrow (5)$ It follows from Proposition 3.4(1).

$(5)\Rightarrow (1)$ It follows from (5) and Proposition 3.4(3).
\end{proof}


From Definition 2.3(3), we note that a G\"{o}del algebra is an MTL-algebra satisfies $x\odot y=x\wedge y$, for all $x,y\in L$. Applying Proposition 3.4(7), one can obtain that the set $X_r$ is closed under the operations $\wedge,\vee,\rightarrow$. Now, we have the following natural questions:
\begin{enumerate}[(1)]
  \item For any nonempty subset $X$ of a G\"{o}del algebra $L$, is $X_r\cup\{0\}$ a subalgebra of $L$?
  \item Whether there exists a nonempty subset $X$ such that $X_r$ is a subalgebra of a G\"{o}del algebra $L$?
\end{enumerate}

For the first question, we have a negative answer as the following example shows:
\begin{example}\emph{Let $L=\{0,a,b,c,d,1\}$, where $0\leq a\leq c\leq 1$, $0\leq b\leq c,d\leq 1$. Define operations $\odot$ and $\rightarrow$ as follows}:\\
\begin{center}
\begin{tabular}{c|c c c c c c}
   $\odot$ & $0$ & $a$ & $b$ & $c$ & $d$ & $1$\\
   \hline
   $0$ & $0$ & $0$ & $0$ & $0$ & $0$ & $0$ \\
   $a$ & $0$ & $a$ & $0$ & $a$ & $0$ & $a$ \\
   $b$ & $0$ & $0$ & $b$ & $b$ & $b$ & $b$ \\
   $c$ & $0$ & $a$ & $b$ & $c$ & $b$ & $c$\\
   $d$ & $0$ & $0$ & $b$ & $b$ & $d$ & $d$ \\
   $1$ & $0$ & $a$ & $b$ & $c$ & $d$ & $1$
 \end{tabular} {\qquad}
\begin{tabular}{c|c c c c c c}
   $\rightarrow$ & $0$ & $a$ & $b$ & $c$ & $d$ & $1$\\
   \hline
   $0$ & $1$ & $1$ & $1$ & $1$ & $1$ & $1$ \\
   $a$ & $d$ & $1$ & $d$ & $1$ & $d$ & $1$ \\
   $b$ & $a$ & $a$ & $1$ & $1$ & $1$ & $1$ \\
   $c$ & $0$ & $a$ & $d$ & $1$ & $d$ & $1$\\
   $d$ & $a$ & $a$ & $c$ & $c$ & $1$ & $1$ \\
   $1$ & $0$ & $a$ & $b$ & $c$ & $d$ & $1$
 \end{tabular}
\end{center}

\emph{Then $(L,\wedge,\vee,\odot,\rightarrow,0,1)$ is a G\"{o}del algebra. Let $X=\{b\}$, then $X_r\cup\{0\}=\{0,a,1\}$ is not a subalgebra of $L$ since $a\rightarrow 0=d\notin X_r\cup\{0\}$.}
\end{example}

\begin{remark}\emph{ It was described that in \cite{Somayeh} that if $L$ is a G\"{o}del algebra and $X$ is a nonempty subset of $L$, then $X_r\cup\{0\}$ is a subalgebra of $L$. From Example 3.7, we know that the above statement is not true. Moreover, Example 3.7 also gives a negative answer to the open problem in \cite{Somayeh} that ``Let $X$ be a nonempty subset of a BL-algebra $L$. Is $X_r\cup\{0\}$ a subalgebra of $L$?" since a G\"{o}del algebra is a subclass of a BL-algebra.}
\end{remark}

As to the second question, we have the following proposition.

\begin{proposition}\emph{Let $L$ be a G\"{o}del algebra and $X$ a nonempty subset of $L$. Then $(L_0)_r$ is a subalgebra of $L$.}
\end{proposition}
\begin{proof} For any $x\in L_0$, we have $x\rightarrow 1=1$ and $x\rightarrow 0=0$ and so $0,1\in (L_0)_r$. By Proposition 3.4(7), we have $(L_0)_r$ is closed under the operations $\wedge,\vee,\rightarrow$. Therefore, $(L_0)_r$ is a subalgebra of $L$.
\end{proof}

From the Proposition 3.4(8), it is natural to ask that whether there exists a nonempty subset $X$ such that $X_l=F$ for given filter $F$ in an MTL-algebra $L$. For this question, we give the positive answer under some suitable conditions by Theorem 3.6 if $F=\langle X\rangle_l$. Furthermore, we have the {\bf  first open problem:}

\begin{enumerate}[(1)]
  \item For any filter $F$ of a MTL-algebra $L$, whether there exists a nonempty subset $X$ such that $X_l=F$?
\end{enumerate}



The following theorem shows that implicative stabilizer $X_s$ is equivalent to  $^\bot X=\{a\in L|a\vee x=1$, for all $x\in X\}$, which was introduced in \cite{Turunen}.

\begin{theorem}\emph{Let $L$ be an MTL-algebra and $X$ a nonempty subset of $L$. Then $^\bot X= X_r\cap X_l=X_s$.}
\end{theorem}
\begin{proof} Let $a\in L$ such that $a\vee x=1$, for all $x\in X$. By Proposition 2.2(10), we have $1=x\vee a=((x\rightarrow a)\rightarrow a)\wedge ((a\rightarrow x)\rightarrow x)$, hence $(x\rightarrow a)\rightarrow a=1$ and $(a\rightarrow x)\rightarrow x=1$. Suppose that $(x\rightarrow a)\rightarrow a=1$, then by Proposition 2.2(11), $a\leq a\rightarrow x$, and so $a\rightarrow(x\rightarrow a)=1$. Thus, $x\rightarrow a=a$, for all $x\in X$. Therefore, $a\in X_r$. In the similar way, we can see that $a\rightarrow x=x$, for all $x\in X$, and hence $a\in X_l$. Therefore, $a\in X_r\cap X_l$. Conversely, if $a\in X_r\cap X_l$, then we have $x\rightarrow a=a$ and $a\rightarrow x=x$, for all $x\in X$. Thus $x\rightarrow a\rightarrow a=1$ and $a\rightarrow x\rightarrow x=1$, for all $x\in X$ and hence $a\vee x=((x\rightarrow a)\rightarrow a)\wedge ((a\rightarrow x)\rightarrow x)=1$, that is, $a\in ^\bot X$. Therefore,$^\bot X=  X_r\cap X_l=X_s$.
\end{proof}



In the following theorems, we give some characterizations of IMTL-algebras and integral MTL-algebras via implicative stabilizers.
\begin{theorem}\emph{Let $L$ be an MTL-algebra. Then the following assertions are equivalent:}
\begin{enumerate}[(1)]
  \item \emph{$L$ is an IMTL-algebra,
  \item $L_0=S_0=^\bot\{0\}$,
  \item if $x\rightarrow y,y\rightarrow x\in L_0$, then $x=y$, for any $x,y\in L$.}
\end{enumerate}
\end{theorem}
\begin{proof} $(1)\Rightarrow(2)$ Let $L$ be an IMTL-algebra and $a\in L_0$. Then $a\rightarrow 0=0$ and so $a=\neg\neg a=(a\rightarrow 0)\rightarrow 0=0\rightarrow 0=1$. Thus, $L_0=\{1\}$. Furthermore, from Proposition 3.4(6) and Theorem 3.9, we have $L_0=S_0=^\bot \{0\}$.

$(2)\Rightarrow(1)$ If $L_0=S_0=^\bot \{0\}$, then $L_0=\{1\}$. Now, for any $x\in L$, we have $x\rightarrow \neg\neg x=x\rightarrow((x\rightarrow 0)\rightarrow 0)=(x\rightarrow 0)\rightarrow (x\rightarrow 0)=1$, and so $\neg(\neg\neg x\rightarrow x)=0$, then $(\neg\neg x\rightarrow x)\rightarrow 0=0$. Thus, $\neg\neg x\rightarrow x\in L_0=\{1\}$. Therefore, $\neg\neg x\rightarrow x=1$ and so $\neg\neg x=x$ for any $x\in L$. Thus, $L$ is an IMTL-algebra.

$(2)\Rightarrow(3)$ Let $x\rightarrow y,y\rightarrow x\in L_0$. Since $L_0=^\bot \{0\}=\{1\}$, then $x\rightarrow y=y\rightarrow x=1$ and so $x=y$.

$(3)\Rightarrow(2)$ Let $x\in L_0$. Since $1\in L_0$, we have $x\rightarrow 1=1\in L_0$ and $1\rightarrow x=x\in L_0$, then by (3), we obtain that $x=1$. Hence $L_0=\{1\}$. Therefore, $L_0=S_0=^\bot \{0\}$.
\end{proof}

\begin{theorem}\emph{Let $L$ be an MTL-algebra. Then the following assertions are equivalent:}
\begin{enumerate}[(1)]
  \item \emph{$L$ is an integral MTL-algebra,
  \item $L_0=L\backslash \{0\}$.}
\end{enumerate}
\end{theorem}
\begin{proof}$(1)\Rightarrow(2)$ Let $L$ be an integral MTL-algebra. Since $0\rightarrow 0=1$ and so $0\notin L_0$. Hence $L_0\subseteq L\backslash\{0\}$. Now, if $x\in L\backslash\{0\}$, from Proposition 2.2(12), we have $x\odot\neg x=0$ and so $\neg(x\odot \neg x)=1$. By Theorem 2.5, we have that $\{1\}$ is an integral filter, then $\neg x=1$ or $\neg\neg x=1$. If $\neg x=1$, the $x\odot \neg x=x\odot 1=x$ and so $x=0$, that is a contradiction. Hence, $\neg\neg x=1$, by Proposition 2.2(10), we have $\neg x=\neg\neg\neg x=0$ and hence $x\in L_0$. Therefore, $L_0=L\backslash\{0\}$.

$(2)\Rightarrow(1)$ Let $x\odot y=0$, for any $x,y\in L$. Then $(x\odot y)\notin L_0=L\backslash\{0\}$. Since $L_0$ is a proper filter, then $x\notin L_0$ or $y\notin L_0$. Therefore, $x=0$ or $y=0$ and so $L$ is an integral MTL-algebra.
\end{proof}

Theorem 3.10(3.11) suggests a method of how to check an MTL-algebra is an IMTL-algebra (integral MTL-algebra). In what follows, we give some examples to show the applications of Theorem 3.10(3.11).

\begin{example}\emph{Let $L=\{0,a,b,c,d,1\}$, where $0\leq a\leq b\leq c\leq d\leq e\leq 1$. Define operations $\odot$ and $\rightarrow$ as follows}:\\
\begin{center}
\begin{tabular}{c|c c c c c c }
   $\odot$ & $0$ & $a$ & $b$ & $c$ & $d$ & $1$\\
   \hline
   $0$ & $0$ & $0$ & $0$ & $0$ & $0$   & $0$ \\
   $a$ & $0$ & $0$ & $0$ & $0$ & $0$   & $a$ \\
   $b$ & $0$ & $0$ & $0$ & $0$ & $b$   & $b$ \\
   $c$ & $0$ & $0$ & $0$ & $a$ & $b$   & $c$\\
   $d$ & $0$ & $0$ & $b$ & $b$ & $d$   & $d$ \\
   $1$ & $0$ & $a$ & $b$ & $c$ & $d$   & $1$
 \end{tabular} {\qquad}
\begin{tabular}{c|c c c c c c }
   $\rightarrow$ & $0$ & $a$ & $b$ & $c$ & $d$  & $1$\\
   \hline
   $0$ & $1$ & $1$ & $1$ & $1$ & $1$   & $1$ \\
   $a$ & $d$ & $1$ & $1$ & $1$ & $1$   & $1$ \\
   $b$ & $c$ & $c$ & $1$ & $1$ & $1$   & $1$ \\
   $c$ & $b$ & $c$ & $d$ & $1$ & $1$   & $1$\\
   $d$ & $a$ & $a$ & $c$ & $c$ & $1$   & $1$ \\
   $1$ & $0$ & $a$ & $b$ & $c$ & $d$   & $1$
 \end{tabular}
\end{center}

\emph{Then $(L,\wedge,\vee,\odot,\rightarrow,0,1)$ is an MTL-algebra. One can check that $L_0=S_0=\{1\}=^\bot \{0\}$, by Theorem 3.10, we know that $(L,\wedge,\vee,\odot,\rightarrow,0,1)$ is an IMTL-algebra.}
\end{example}

\begin{example}\emph{Let $L=\{0,a,b,c,1\}$ with $0\leq c\leq a,b\leq 1$. Consider the operations $\odot$ and $\rightarrow$ given by the following tables}:\\
\begin{center}
\begin{tabular}{c|ccccc}

  $\odot$ & $0$ & $a$ & $b$ & $c$ & $1$ \\
   \hline
  $0$ & $0$ & $0$ & $0$ & $0$ & $0$ \\
  $c$ & $0$ & $c$ & $c$ & $c$ & $c$ \\
  $a$ & $0$ & $c$ & $a$ & $c$ & $a$ \\
  $b$ & $0$ & $c$ & $c$ & $b$ & $b$ \\
  $1$ & $0$ & $a$ & $b$ & $c$ & $1$ \\
\end{tabular}{\qquad}
\begin{tabular}{c|ccccc}

  $\rightarrow$ & $0$ & $a$ & $b$ & $c$ & $1$ \\
   \hline
  $0$ & $1$ & $1$ & $1$ & $1$ & $1$ \\
  $c$ & $0$ & $1$ & $1$ & $1$ & $1$ \\
  $a$ & $0$ & $b$ & $1$ & $b$ & $1$ \\
  $b$ & $0$ & $a$ & $a$ & $1$ & $1$ \\
  $1$ & $0$ & $a$ & $b$ & $c$ & $1$ \\
\end{tabular}
\end{center}

\emph{Then $(L,\wedge,\vee,\odot,\rightarrow,0,1)$ is an MTL-algebra. One can check that $L_0=\{c,a,b,1\}\\=L\backslash \{0\}$, by Theorem 3.11, we know that $(L,\wedge,\vee,\odot,\rightarrow,0,1)$ is an integral MTL-algebra.}
\end{example}

 The following theorem shows that the left and right implicative stabilizers are equivalent in any MV-algebras.

\begin{theorem}\emph{Let $L$ be an MTL-algebra. If $L$ is an MV-algebra, then $X_l=X_r=X_s=^\bot X$ for any nonempty subset $X$ of $L$.}
\end{theorem}
\begin{proof} We note that an MV-algebra satisfies $(x\rightarrow y)\rightarrow y=(y\rightarrow x)\rightarrow x$ for all $x,y\in L$. Now, we will prove $X_l=X_r$. Let $a\in X_l$, then $a\rightarrow x=x$ for all $x\in X$ and hence $(a\rightarrow x)\rightarrow x=1$. On one hand, since $L$ is an MV-algebra and so $(x\rightarrow a)\rightarrow a=1$. On the other hand, from Proposition 2.2(9), we have $a\leq x\rightarrow a$. Combining them, one can obtain that $x\rightarrow a=a$. Thus, $a\in X_l$. Hence $X_l\subseteq X_r$. In the similar way, we have $X_r\subseteq X_l$. Therefore, $X_l=X_r$. From Theorem 3.10, we can see that $X_l=X_r=X_s=^\bot X$.
\end{proof}

From the above theorems, we give some characterizations of MV-algebra via implicative stabilizers.

\begin{theorem}\emph{Let $L$ be a BL-algebra. Then the following assertions are equivalent:}
\begin{enumerate}[(1)]
  \item \emph{$L$ is an MV-algebra,
  \item $X_l=X_r=X_s=^\bot X$,for any nonempty subset $X$ of $L$,
  \item $L_0=S_0=^\bot \{0\}$,
  \item $X_l=\langle X\rangle_l$, for any nonempty subset $X$ of $L$,
  \item $a\rightarrow b=b$ if and only if $b\rightarrow a=a$, for any $a,b\in L$,
  \item $X_r$ is a filter of $L$, for any nonempty subset $X$ of $L$,
  \item $S_x=R_x=L_x$, for all $x\in L$.}
\end{enumerate}
\end{theorem}
\begin{proof} $(1)\Rightarrow (2)$ It follows from Theorem 3.15.

$(2)\Rightarrow(3)$ Taking $X=\{0\}$ in (2).

$(3)\Rightarrow(1)$ It follows from Theorem 3.11.

$(4)\Rightarrow (5)$, $(5)\Rightarrow (6)$, $(6)\Rightarrow (7)$, $(7)\Rightarrow (4)$ follow from Theorem 3.6.
\end{proof}

\begin{proposition}\emph{Let $L$ be a BL-algebra and $F$ be any filter of $L$. If $F=(\langle F\rangle_r)_r$, then $L$ is an MV-algebra.}
\end{proposition}
\begin{proof} By Proposition 3.4(3) and hypothesis, we have $\langle a\rangle=(\langle a\rangle_r)_r=(R_a)_r$. For any $x\in R_a$, since $a\in\langle a\rangle=(R_a)_r$, we obtain $y\rightarrow a=a$ for any $y\in R_a$. Thus $x\rightarrow a=a$, that is, $x\in L_a$. Therefore, $R_a\subseteq L_a$. On the other hand, for any $x\in L_a$, we have $x\rightarrow a=a$, thus $a\in R_x$ and hence $a\in L_x$, which implies that $a\rightarrow x=x$, that is $x\in R_a$. Based on the above, we get that $L_a=R_a$, for all $a\in L$. From Theorem 3.16 $(7)\Rightarrow(1)$, we get that $L$ is an MV-algebra.
\end{proof}

The following example shows that the converse of the above proposition is not true, in general.

\begin{example}\emph{Let $L=\{0,a,b,c,d,1\}$, where $0\leq b\leq a\leq 1$, $0\leq d\leq a,c\leq 1$. Define operations $\odot$ and $\rightarrow$ as follows}:\\
\begin{center}
\begin{tabular}{c|c c c c c c}
   $\odot$ & $0$ & $a$ & $b$ & $c$ & $d$  & $1$\\
   \hline
   $0$ & $0$ & $0$ & $0$ & $0$ & $0$  & $0$ \\
   $a$ & $0$ & $b$ & $b$ & $d$ & $0$  & $a$ \\
   $b$ & $0$ & $b$ & $b$ & $0$ & $0$  & $b$ \\
   $c$ & $0$ & $d$ & $0$ & $c$ & $d$  & $c$\\
   $d$ & $0$ & $0$ & $0$ & $d$ & $0$  & $d$ \\
   $1$ & $0$ & $a$ & $b$ & $c$ & $d$  & $1$
 \end{tabular} {\qquad}
\begin{tabular}{c|c c c c c c}
   $\rightarrow$ & $0$ & $a$ & $b$ & $c$ & $d$  & $1$\\
   \hline
   $0$ & $1$ & $1$ & $1$ & $1$ & $1$  & $1$ \\
   $a$ & $d$ & $1$ & $a$ & $c$ & $c$  & $1$ \\
   $b$ & $c$ & $1$ & $1$ & $c$ & $c$  & $1$ \\
   $c$ & $b$ & $a$ & $b$ & $1$ & $a$  & $1$\\
   $d$ & $a$ & $1$ & $a$ & $1$ & $1$  & $1$ \\
   $1$ & $0$ & $a$ & $b$ & $c$ & $d$  & $1$
 \end{tabular}
\end{center}

\emph{One can easily check that $(L,\wedge,\vee,\odot,\rightarrow,0,1)$ is an MV-algebra and $F=\{a,1\}$ is a filter of $L$. However, $(F_r)_r=\{a,b,1\}\neq F$.}
\end{example}

\begin{remark}\emph{ It was proved that in \cite{turunen} that RS-BL-algebras, which were defined in \cite{Somayeh}, are equivalent to MV-algebras. Therefore, Example 3.18 gives a negative answer to the another open problem in \cite{Somayeh} that `` Is the converse of the Proposition 3.17 true?".}
\end{remark}

As we known, BL-algebras are a class of residuated structure based on continuous t-norm and its residua, then it satisfies the condition of divisibility $x\wedge y=x\odot (x\rightarrow y)$, in this case, the condition $\neg\neg x=x$ is equivalent to the $(x\rightarrow y)\rightarrow y=(y\rightarrow x)\rightarrow x$ in this algebraic structure.  Using this important result, one can prove that a BL-algebra $L$ is an MV-algebra if and only if $X_l=X_r$ for any nonempty subset $X$ of $L$ from Theorem 3.16. Compared to BL-algebras, MTL-algebras are a more general residuated structure based on left-continuous t-norm and its residua, which not satisfies the condition of divisibility.  Based on the above consideration, we have the {\bf second open problem}:
\begin{enumerate}[(2)]
  \item Let $L$ be an MTL-algebra and $X_l=X_r$ for any nonempty subset $X$ of $L$. Is $L$ an MV-algebra?
\end{enumerate}

\section{Multiplicative stabilizers in MTL-algebras}
In this section, we investigate related properties of multiplicative stabilizers and discuss the relations between implicative stabilizers and  multiplicative stabilizers. Also, we prove that the left and right multiplicative stabilizers form two MTL-algebras. Finally, using multiplicative stabilizers, we give some characterizations of G\"{o}del algebras and linearly order G\"{o}del algebras.

\begin{definition}\emph{Let $L$ be an MTL-algebra and $X$ be a nonempty subset of $L$. The \emph{left and right multiplicative stabilizer} of $X$ are defined as follow, respectively},
\begin{center}$X^\ast_l=\{a\in L|a\odot x=x$, \emph{for all} $x\in X\}$,\\
$X^\ast_r=\{a\in L|x\odot a=a$, \emph{for all} $x\in X\}$.
\end{center}

\emph{The set $X^\ast_s=X^\ast_l\cap X^\ast_r$ is called the \emph{multiplicative stabilizer} of $X$. For convenience, the multiplicative stabilizer, left multiplicative stabilizer and right multiplicative stabilizer of $X=\{x\}$ are denoted by $S^\ast_x$, $L^\ast_x$ and $R^\ast_x$, respectively.}
\end{definition}
\begin{example}\emph{Considering the MTL-algebra in Example 3.2. If $X=\{a,b\}$, then $X^\ast_l=\{1,b\}$, $X^\ast_r=\{0,b\}$ and hence $X^\ast_s=\{b\}$. In this case, $X^\ast_r$ is not a filter of $L$.}
\end{example}

\begin{proposition} \emph{Let $L$ be an MTL-algebra and $X,Y$ be two nonempty subsets of $L$. Then the following assertions hold:}
\begin{enumerate}[(1)]
   \item \emph{$X^\ast_r=\cap_{x\in X}R^\ast_x$, $X^\ast_l=\cap_{x\in X}L^\ast_x$ and $X^\ast_s=\cap_{x\in X}S^\ast_x$,
  \item If $X\subseteq Y$, then $Y^\ast_r\subseteq X^\ast_r$, $Y^\ast_l\subseteq X^\ast_l$, and $Y^\ast_s\subseteq X^\ast_s$,
  \item $\langle X\rangle^\ast_r=X^\ast_r$,
  \item $X=\{0\}$ if and only if $X^\ast_l=L$, $X^\ast_r=X^\ast_s=\{0\}$,
  \item $R^\ast_1=\{1\}$ and $L^\ast_1=\{1\}$ and so $S^\ast_1=\{1\}$,
  \item $X^\ast_s=X$,
  \item $X^\ast_l$ is a filter of $L$,
  \item if $a,b\in X^\ast_r$ ($X^\ast_l$), then $a\vee b,a\odot b\in X^\ast_r$($X^\ast_l$),
  \item $X_r\cap X^\ast_r\subseteq \{a|x\odot(x\rightarrow a)=a\odot x=a$ for all $x\in X\}$, $X_l\cap X^\ast_l\subseteq \{a|a\odot(a\rightarrow x)=a\odot x=x$ for all $x\in X\}$,
  \item if $L$ is a BL-algebra, then $X_r\cap X^\ast_r\subseteq \{a|x\wedge a=a$, for all $x\in X\}=\{a|a\in (x]$, for all $x\in X\}$, $X_l\cap X^\ast_l\subseteq \{a|a\wedge x=x$ for all $x\in X\}=\{a|a\in [x)$, for all $x\in X\}$,
  \item if $X\subseteq G(L)$, then $X_r\cap X^\ast_r\subseteq G(X)$,  $X_l\cap X^\ast_l\subseteq G(X)$ and hence $X_s\cap X^\ast_s\subseteq G(X)$.}
\end{enumerate}
\end{proposition}
\begin{proof} The proofs of (1),(2),(4),(5),(6) and (7) are similar to that of Proposition 3.4.

(3) Since $X\subseteq \langle X\rangle$, by (2), we have $\langle X\rangle^\ast_r\subseteq X^\ast_r$. On the other hand, suppose that $a\in X^\ast_r$, so $a\odot x=a$, for all $x\in X$. For any $y\in \langle X\rangle$, there exist $y_1,y_2,\cdots y_n\in X$ such that $y_1\odot y_2\cdots \odot y_n\leq y$ and hence $a\odot y\geq a\odot y_1\odot y_2\cdots \odot y_n=a$. Moreover, we have $a\odot y\leq a$, for all $y\in \langle X\rangle$, and so $a\in \langle X\rangle^\ast_r$. Therefore, $\langle X\rangle^\ast_r=X^\ast_r$.

(8) For all $a,b\in X^\ast_r$, we have $a=x\odot a$ and $b=x\odot b$. It follows that $a\odot b=(x\odot a)\odot(x\odot b)\leq x\odot(a\odot b)\leq a\odot b$, which implies that $a\odot b=x\odot(a\odot b)$. Thus, $a\odot b\in X^\ast_r$. On the other hand, we have $a\vee b=(a\odot x)\vee (b\odot x)=x\odot(a\vee b)$ for all $a,b\in X^\ast_r$. Hence we can obtain $x\odot(a\vee b)=a\vee b$, that is, $a\vee b\in X^\ast_r$. One can also check that $X^\ast_l$ is closed under the operations $\odot$ and $\vee$.

The proofs of (9),(10),(11) are easy and hence we omit them.
\end{proof}

The following example shows that $X^\ast_l=\langle X\rangle^\ast_l$ not holds, in general.

\begin{example}\emph{Let $L=\{0,a,b,1\}$ with $0\leq a\leq b\leq 1$. Consider the operations $\odot$ and $\rightarrow$ given by the following tables}:
\begin{center}
\begin{tabular}{c|cccc}

  $\odot$ & $0$ & $a$ & $b$ & $1$ \\
   \hline
  $0$ & $0$ & $0$ & $0$ & $0$ \\
  $a$ & $0$ & $0$ & $a$ & $a$ \\
  $b$ & $0$ & $a$ & $b$ & $b$ \\
  $1$ & $0$ & $a$ & $b$ & $1$ \\
\end{tabular}{\qquad}
\begin{tabular}{c|cccc}

  $\rightarrow$ & $0$ & $a$ & $b$ & $1$ \\
   \hline
  $0$ & $1$ & $1$ & $1$ & $1$ \\
  $a$ & $a$ & $1$ & $1$ & $1$ \\
  $b$ & $0$ & $a$ & $1$ & $1$ \\
  $1$ & $0$ & $a$ & $b$ & $1$ \\
\end{tabular}
\end{center}

\emph{Then $(L,\wedge,\vee,\odot,\rightarrow,0,1)$ is an MTL-algebra. If $X=\{b\}$, then $X^\ast_l=\{b,1\}$ and $\langle X \rangle^\ast_l=\{1\}$ and hence $X^\ast_l\neq \langle X\rangle^\ast_l$.}
\end{example}

From Proposition 4.3(8), we know that  $X^\ast_r$ is closed under $\vee$, in general, the following example shows that $X^\ast_r$ is not a lattice ideal of $L$.

\begin{example}\emph{Let $L=\{0,a,b,c,1\}$ with $0\leq a\leq b\leq c\leq 1$. Consider the operation $\odot$ and $\rightarrow$ given by the following tables}:\\
\begin{center}
\begin{tabular}{c|ccccc}

  $\odot$ & $0$ & $a$ & $b$ & $c$ & $1$ \\
   \hline
  $0$ & $0$ & $0$ & $0$ & $0$ &$0$ \\
  $a$ & $0$ & $a$ & $a$ & $a$ &$a$ \\
  $b$ & $0$ & $a$ & $a$ & $b$ &$b$ \\
  $c$ & $0$ & $a$ & $a$ & $c$ &$c$ \\
  $1$ & $0$ & $a$ & $b$ & $c$ & $1$ \\
\end{tabular}{\qquad}
\begin{tabular}{c|ccccc}

  $\rightarrow$ & $0$ & $a$ & $b$ & $c$ & $1$ \\
   \hline
  $0$ & $1$ & $1$ & $1$ & $1$ & $1$ \\
  $a$ & $a$ & $1$ & $1$ & $1$ & $1$ \\
  $b$ & $0$ & $c$ & $1$ & $1$ & $1$ \\
  $c$ & $0$ & $b$ & $b$ & $1$ & $1$ \\
  $1$ & $0$ & $a$ & $b$ & $c$ & $1$ \\
\end{tabular}
\end{center}

\emph{Then $(L,\wedge,\vee,\odot,\rightarrow,0,1)$ is an MTL-algebra, which does not satisfy the divisibility condition $x\wedge y=x\odot(x\rightarrow y)$, for all $x,y,z\in L$. Let $X=\{a,c\}$, one can check that $X^\ast_r=\{0,a,c\}$ is not a lattice ideal of $L$, since $c\in X^\ast_r$ and $b\leq c$, but $b\notin X^\ast_r$.}
\end{example}

Then following proposition shows that if $L$ is a BL-algebra, we can obtain that $X^\ast_r$ is a lattice ideal of $L$.

\begin{proposition}\emph{Let $L$ be a BL-algebra. Then $X^\ast_r$ is a lattice ideal of $L$.}
\end{proposition}
\begin{proof} From Proposition 4.3(8), we conclude that $X^\ast_r$  is closed under $\vee$, that is, $a\vee b\in X^\ast_r$, for all $a,b\in X^\ast_r$. Moreover, for all $a\in X^\ast_r$ and $b\in L$, let $b\leq a$, we will show $b\in X^\ast_r$. It follows from $b\leq a$ that $b=b\wedge a$. Since $x\odot b=x\odot(a\wedge b)=x\odot(a\odot(a\rightarrow b))=(x\odot a)\odot(a\rightarrow b)=a\odot(a\rightarrow b)=a\wedge b=b$. Thus, $b\in X^\ast_r$.  Therefore, $X^\ast_r$ is a lattice ideal of $L$.
\end{proof}

The following theorems show that the $R^\ast_x$ and $L^\ast_x$ have the same structure as an MTL-algebra under which suitable conditions, which reveals the essence of the stabilizer in MTL-algebras.

\begin{theorem}\emph{Let $L$ be an MTL-algebra and $x$ be an idempotent element of $L$. Then $(L^\ast_x, \odot,\wedge,\vee,\rightarrow,x,1)$ is an MTL-algebra.}
\end{theorem}
\begin{proof} First, we show that $(L^\ast_x,\wedge,\vee,x,1)$ is a bounded lattice with $x$ as the smallest element
and $1$ as the greatest element. From Proposition 4.3(8), we have that $L^\ast_x$ is closed under $\vee$. On the other hand, we have $x\odot(a\wedge b)=x\odot (a\wedge b)=x\wedge(a\wedge b)=(x\wedge a)\wedge(x\wedge b)=(x\odot a)\wedge(x\odot b)=x\wedge x=x$, that is, $a\wedge b\in L^\ast_x$.
Therefore $(L^\ast_x,\wedge,\vee)$ is a lattice. For all $a\in L^\ast_x$, one can easily check that $a\wedge x=a\odot x=x$ and $a\vee 1=1$. Thus, $x$ is the smallest element and $1$ is the greatest element in $L^\ast_x$, respectively.

Next, we prove that $(L^\ast_x,\odot,1)$ is a commutative monoid with $1$ as neutral
element. From Proposition 4.3(8), we have $L^\ast_x$ is closed under $\odot$.  It follows that $(L^\ast_x,\odot)$ is a commutative semigroup. For all $a\in L^\ast_x$, we can obtain that
$a\odot 1=a$, that is, $1$ is a unital element.

From Proposition 2.4(2), for all $a,b\in L^\ast_x$, we have $x\odot(a\rightarrow b)=x\odot[(x\odot a)\rightarrow (x\odot b)]=x\odot(x\rightarrow x)=x$, that is, $a\rightarrow b\in L^\ast_x$.

Therefore, we obtain that $(L^\ast_x, \odot,\wedge,\vee,\rightarrow,x,1)$ is an MTL-algebra.
\end{proof}

\begin{theorem}\emph{Let $L$ be an MTL-algebra and $x$ be an idempotent element of $L$. Then $(R^\ast_x, \odot,\wedge,\vee,\rightsquigarrow,0,x)$ is an MTL-algebra, where $a\rightsquigarrow b=x\odot(a\rightarrow b)$, for all $a,b\in L$.}
\end{theorem}
\begin{proof} First, we show that $(R^\ast_x,\wedge,\vee,0,x)$ is a bounded lattice with 0 as the smallest element
and $x$ as the greatest element. From Proposition 4.3(8), we have that $R^\ast_x$ is closed under $\vee$. On the other hand, we have $x\odot(a\wedge b)=x\wedge (a\wedge b)=(x\wedge a)\wedge(x\wedge b)=(x\odot a)\wedge(x\odot b)=a\wedge b$, that is, $a\wedge b\in R^\ast_x$.
Therefore $(R^\ast_x,\wedge,\vee)$ is a lattice. For all $a\in R^\ast_x$, one can easily check that $a\vee x=(x\odot a)\vee x=x$ and $a\wedge 0=0$. Thus, 0 is the smallest element and $x$ is the greatest element in $R^\ast_x$, respectively.

Next, we prove that $(R^\ast_x,\odot,x)$ is a commutative monoid with $x$ as neutral
element. From Proposition 4.3(8), we have $R^\ast_x$ is closed under $\odot$.  It follows that $(R^\ast_x,\odot)$ is a commutative semigroup. For all $a\in R^\ast_x$, we can obtain that
$a\odot x=a$, that is, $x$ is a unital element.

For all $a,b\in R^\ast_x$, we define $a\rightsquigarrow b=x\odot(a\rightarrow b)$. Now, we will show that $a\odot b\leq c$ if and only if $b\leq a\rightsquigarrow c$ for
all $a,b,c\in R^\ast_x$. One can easily check that $x\odot a\leq b$ if and only if $x\odot a\leq x\odot b$. Hence we have $a\odot b\leq c$ if and only if
$b\leq a\rightarrow c$ if and only if $x\odot b\leq a\rightarrow c$ if and only if $x\odot b\leq x\odot(a\rightarrow c)$ if and only if $x\odot b\leq a\rightsquigarrow b$ if and only if $b\leq a\rightsquigarrow c$ for
all $a,b,c\in R^\ast_x$.

For all $a,b\in R^\ast_x$, we have $(a\rightsquigarrow b)\vee(b\rightsquigarrow a)=(x\odot (a\rightarrow b))\vee (x\odot (b\rightarrow a))=x\odot((a\rightarrow b)\vee(b\rightarrow a))=x\odot 1=x$.

Therefore, we obtain that $(R^\ast_x, \odot,\wedge,\vee,\rightsquigarrow,0,x)$ is an MTL-algebra.
\end{proof}

From Theorems 4.7 and 4.8, we obtain that $(R^\ast_x, \odot,\wedge,\vee,\rightsquigarrow,0,x)$ and $(L^\ast_x,
 \odot,\wedge,\vee,\\ \rightarrow,x,1)$ form two MTL-algebras, respectively. Now, we have the {\bf third open problem}:

\begin{enumerate}[(3)]
  \item Are $(R^\ast_x, \odot,\wedge,\vee,\rightsquigarrow,0,x)$ and $(L^\ast_x,\odot,
  \wedge,\vee, \rightarrow,x,1)$ isomorphic?
\end{enumerate}

In what follows, we give some characterizations of G\"{o}del algebras and linearly ordered G\"{o}del algebras via multiplicative stabilizers.
 \begin{theorem}\emph{Let $L$ be an MTL-algebra. Then the following assertions are equivalent:}
 \begin{enumerate}[(1)]
   \item \emph{$L$ is a G\"{o}del algebra,
   \item for all $x\in L$, $L^\ast_x=\langle x\rangle=[x)$,
   \item for all $x\in L$, $R^\ast_x=(x]$.}
 \end{enumerate}
\end{theorem}
\begin{proof} $(1)\Rightarrow (2)$ Assume that $L$ is a G\"{o}del algebra, we have that $x\odot x=x$ for all $x\in L$. Thus, $x\in L^\ast_x$. Moreover, we have that $L^\ast_x$ is a filter of $L$ by combining Proposition 4.3(7). That is, for all $x\in L$, if $x\leq a$, we can obtain $a\in L^\ast_x$, which implies that $[x)\subseteq L^\ast_x$. Next, we will show that $L^\ast_x\subseteq [x)$. For all $a\in L^\ast_x$, we have $x\wedge a=x\odot a=x$, which implies that $x\leq a$, that is, $a\in [x)$. It follows that $L^\ast_x=[x)$. Therefore, we obtain $L^\ast_x=[x)$. Moreover, one can easily check that if $x$ is an idempotent element of $L$, then $\langle x\rangle=[x)$.

$(2)\Rightarrow (1)$ Suppose that $L^\ast_x=[x)$ for all $x\in L$. Since $x\in[x)=
L^\ast_x$, that is, $x\odot x=x$ for all $x\in L$. Next, we will prove
that $x\odot y = x\wedge y = x\odot(x\rightarrow y)$ for all $x,y\in L$. For all $x,y\in L$, we have $x\odot(x\rightarrow y)\leq(x\odot x)\rightarrow(x\odot y)$, that is, $x\odot(x\rightarrow y)\leq x\rightarrow(x\odot y)$.
Then we have $x\rightarrow y\leq x \rightarrow(x\rightarrow(x\odot y))=(x\odot x)\rightarrow(x\odot y)= x\rightarrow(x\odot y)$. It
follows that $x\odot(x\rightarrow y)\leq x\odot y$. On the other hand, from $y \leq x\rightarrow y$, we obtain $x\odot y=x\odot(x\rightarrow y)$. Thus, we get $x\odot y=x\odot(x\rightarrow y)$. Finally, we will show $x\odot y=x\wedge y$. For all $x,y\in L$, we have $x\odot y\leq x\wedge y$ by Proposition 2.2(2). Now, for all $u\in L$, if $u\leq x$ and $u\leq y$, we can obtain $u\odot u\leq x\odot y$. Hence we have $u\leq x\odot y$. It follows that $x\odot y=x\wedge y =x\odot(x\rightarrow y)$. Using Definition 2.3(3), we have that $L$ is a G\"{o}del algebra.

$(1)\Rightarrow(3)$ Assume that $L$ is a G\"{o}del algebra, we have that $x\odot x=x$ for all $x\in L$. It follows that $x\odot x=x$ for all $x\in L$. Thus, $x\in R^\ast_x$. Moreover, we have that $R^\ast_x$ is a lattice ideal of $L$ by combining Proposition 4.6. That is, for all $x\in L$, if $a\leq x$, we can obtain $a\in R^\ast_x$, which implies that $(x]\subseteq R^\ast_x$. Next, we will show that $R^\ast_x\subseteq (x]$. For all $a\in R^\ast_x$, we have $x\wedge a=x\odot a=a$, which implies that $a\leq x$, that is, $a\in (x]$. It follows that $R^\ast_x=(x]$. Therefore, we obtain $R^\ast_x=(x]$.

$(3)\Rightarrow(1)$ The proof is similar to $(2)\Rightarrow (1)$.
\end{proof}

\begin{theorem} \emph{Let $L$ be an MTL-algebra. Then the following assertions are equivalent:}
 \begin{enumerate}[(1)]
   \item \emph{$L$ is a linearly ordered G\"{o}del algebra,
   \item for all $x\in L$, $L^\ast_x=[x)$ and $L^\ast_x$ is a prime filter,
   \item for all $x\in L$, $R^\ast_x=(x]$ and $R^\ast_x$ is a lattice prime ideal.}
 \end{enumerate}
\end{theorem}
\begin{proof}$(1)\Rightarrow(2)$ Assume that $L$ is a linearly ordered G\"{o}del algebra, using Theorem 4.9 $(1)\Rightarrow (2)$, we have $L^\ast_x=[x)$ is a filter. Moreover, if $x\vee y\in L^\ast_x$. Since $L$ is a linearly ordered G\"{o}del algebra, then $x\leq y$ or $y\leq x$. Assume $x\leq y$, then $x=x\odot(x\vee y)=x\odot y$. It follows that $y\in L^\ast_x$. This shows that $L^\ast_x$ is a prime filter.

$(2)\Rightarrow(1)$ First, from Theorem 4.9 $(2)\Rightarrow(1)$, we obtain that $L$ is a G\"{o}del algebra. Now, we will prove that $L$ is a linearly ordered G\"{o}del algebra. For $x,y\in L$, consider $L^\ast_{x\vee y}$, which is induced by $x\vee y$. Then $L^\ast_{x\vee y}$ is a prime filter by hypothesis. Note that $x\vee y\in L^\ast_{x\vee y}$. Hence $x\in L^\ast_{x\vee y}$ or $y\in L^\ast_{x\vee y}$. Assume that $x\in L^\ast_{x\vee y}$, then $x\vee y=x\odot(x\vee y)=x\wedge(x\vee y)=x$. So $y\leq x$. This means that $L$ is a linearly ordered G\"{o}del algebra.

$(1)\Rightarrow(3)$ Assume that $L$ is a linearly ordered G\"{o}del algebra, using Theorem 4.9 $(1)\Rightarrow (3)$, we have $R^\ast_x=(x]$ is a lattice ideal. Moreover, if $x\wedge y\in R^\ast_x$. Since $L$ is a  linearly ordered G\"{o}del algebra, then $x\leq y$ or $y\leq x$. Assume $x\leq y$, then $x=x\wedge y=x\odot(x\wedge y)=x\odot x$. It follows that $x\in R^\ast_x$. This shows that $R^\ast_x$ is a prime lattice ideal.

$(3)\Rightarrow(1)$ First, from Theorem 4.9 $(3)\Rightarrow(1)$, we obtain that $L$ is a G\"{o}del algebra. Now, we will prove that $L$ is a linearly ordered G\"{o}del algebra. For $x,y\in L$, consider the $R^\ast_{x\wedge y}$, which induced by $x\wedge y$. Then $R^\ast_{x\wedge y}$ is a lattice prime ideal by hypothesis. Note that $x\wedge y\in R^\ast_{x\wedge y}$. Hence $x\in R^\ast_{x\wedge y}$ or $y\in R^\ast_{x\wedge y}$. Assume that $x\in R^\ast_{x\wedge y}$, then $x=x\odot(x\wedge y)=x\wedge(x\wedge y)$. So $x\leq y$. This means that $L$ is a linearly ordered G\"{o}del algebra.
\end{proof}

    In the following, we discuss the relation between implicative stabilizers and  multiplicative stabilizers. In particular, we shows the sets $R_x (L_x)$ and $R^\ast_x$ are order isomorphic.

 \begin{theorem}\emph{Let $L$ be an MTL-algebra and $x$ be an idempotent element of $L$. Then sets $R_x$ and $R^\ast_x$ are order isomorphic.}
 \end{theorem}
 \begin{proof}For all $x\in L$, let $g:R^\ast_x\rightarrow R_x$ be defined by $g(a)=x\rightarrow a$ for all $a\in R^\ast_x$. Clearly, $g$ is a map from $R^\ast_x$ to $R_x$, that is, $g$ is well defined.
 \begin{enumerate}[(1)]
 \item For all $a,b\in R^\ast_x$, that is, $a=x\odot a$ and $b=x\odot b$, if $g(a)=g(b)$, then $x\rightarrow a=x\rightarrow b$. From $x\odot a\leq a$, we have $a\leq x\rightarrow a$, then $a\leq x\rightarrow b$. It follows that $x\odot a\leq b$, which means $a\leq b$. Similarly, we can prove $b\leq a$. Then $a=b$. Consequently, we obtain that $g$ is injective.
 \item In order to prove that $g$ is surjective, we shall prove the fact that $x\rightarrow (x\odot y)=y$ if and only if there exists $z\in L$ such that $x\rightarrow z=y$ for all $x,y\in L$. Indeed, the fact that $x\rightarrow (x\odot y)=y$ implies that there exists $z\in L$
such that $x\rightarrow z=y$ is obvious. Conversely, if $y= x\rightarrow z$, then $x\odot y\leq z$. It follows that $x\rightarrow (x\odot y)\leq x\rightarrow z=y$. Combining $y\leq x\rightarrow x\odot y$, we have $x\rightarrow (x\odot y)=y$. Now, for all $a\in R_x$, then $a = g(x) =x\rightarrow a$. Using the above result, we have that $g(x\odot a)= x\rightarrow(x\odot a)=x\rightarrow(x\odot(x\rightarrow a))= x \rightarrow a=a$. Thus, we conclude that $g$ is surjective.
 \item For all $a,b\in R^\ast_x$, if $a\leq b$, then $g(a)= x\rightarrow a\leq x\rightarrow b = g(b)$. Therefore, $g$ is order-preserving. Moreover, the inverse map
     $g^{-1}:R^\ast_x\rightarrow R_x$ is also order-preserving, where $g^{-1}(a) =x\odot a$ for all $a\in R_x$.

  Combining them, we obtain that $g$ is an order isomorphism from $R^\ast_x$ to $R_x$. Therefore, the ordered sets $R^\ast_x$ and $R_x$ are isomorphic.
 \end{enumerate}
 \end{proof}

\begin{theorem}\emph{Let $L$ be an MV-algebra and $x$ be an idempotent element of $L$. Then sets $L_x$ and $R^\ast_x$ are order isomorphic.}
 \end{theorem}

\begin{proof} It follows from Theorems 3.15 and 4.11.
\end{proof}
\section{Conclusions}
The aim of this paper is to develop the stabilizer theory of MTL-algebras. In the paper, some useful properties of particular stabilizers are discussed.  And, we characterize some special class of MTL-algebras, for example, IMTL-algebras, integral MTL-algebra, G\"{o}del algebras and MV-algebras, via these stabilizers. Finally, we discuss the relation between these stabilizers and obtain that the right implicative stabilizers are isomorphic to the right multiplicative stabilizers. There are still three open problems:
\begin{enumerate}[(1)]
           \item For any filter $F$ of an MTL-algebra $L$, whether there exists a nonempty subset $X$ such that $X_l=F$?
           \item Let $L$ be an MTL-algebra and $X_l=X_r$ for any nonempty subset $X$ of $L$. Is $L$ an MV-algebra?
           \item  Are $(R^\ast_x, \odot,\wedge,\vee,\rightsquigarrow,0,x)$ and $(L^\ast_x,\odot,
  \wedge,\vee, \rightarrow,x,1)$  isomorphic?
         \end{enumerate}
In our future work, we will consider these problems.

\medskip
\noindent\textbf{Acknowledgments}
\medskip

\indent This research is partially supported by a grant of
National Natural Science Foundation of China (11601302).

\section*{References}

\end{document}